\documentclass[1p,number]{elsarticle}
\usepackage{makeidx,hyperref}
\usepackage{url}
\usepackage{graphicx}
\usepackage{amssymb,amsmath,amsthm}
\usepackage{natbib}

\newtheorem{theorem}{Theorem}
\newtheorem{lemma}{Lemma}
\newtheorem{proposition}{Proposition}

\usepackage{epigraph}

\begin{document}

\begin{frontmatter}

\title{Preservation of dissipativity in dimensionality reduction}
\author[1]{Sergey V. Stasenko\corref{cor1}}\ead{stasenko@neuro.nnov.ru}
\author[2]{Alexander N. Kirdin}\ead{kirdinalexandergp@gmail.com}
\address[1]{Artificial Intelligence Research Center, Lobachevsky University, Nizhny Novgorod, Russia}
\address[2]{Institute of Computational Modelling, Russian Academy of Sciences, Siberian Branch, Krasnoyarsk}
\cortext[cor1]{Corresponding author}

\begin{abstract}
Systems with predetermined Lyapunov functions play an important role in many areas of applied mathematics, physics and engineering: dynamic optimization methods (objective functions and their modifications), machine learning (loss functions), thermodynamics and kinetics (free energy and other thermodynamic potentials), adaptive control (various objective functions, stabilization quality criteria and other Lyapunov functions). Dimensionality reduction is one of the main challenges in the modern era of big data and big models. Dimensionality reduction for systems with Lyapunov functions requires {\it preserving dissipativity}: the reduced system must also have a Lyapunov function, which is expected to be a restriction of the original Lyapunov function on the manifold of the reduced motion. An additional complexity of the problem is that the equations of motion themselves are often unknown in detail in advance and must be determined in the course of the study, while the Lyapunov function could be determined based on incomplete data. Therefore, the projection problem arises: for a given Lyapunov function, find a field of projectors such that the reduction of {\it any} dissipative system is again a dissipative system. In this paper, we present an explicit construction of such projectors and prove their uniqueness. We have also taken the first step beyond the approximation by manifolds. This is required in many applications. For this purpose, we introduce the concept of monotone trees and find a projection of dissipative systems onto monotone trees that preserves dissipativity.
\end{abstract}

\date{}

\end{frontmatter}

\section{Introduction}

Dimensionality reduction is one of the main tasks in data analysis and mathematical modelling of complex system \cite{GorYab2013}. Many linear and non-linear methods were developed for solving this problem. Famous Principal Component Analysis (PCA) was invented by K. Pearson in 1901 \cite{Pearson1901} for data modelling and extraction essential features. Despite its very long history, this method still attracts a lot of attention from many researchers with the publication of a large number of practical recommendations and modifications \cite{GreenacrePCA2022}.

Manifold learning is the first non-linear generalization of PCA   \cite{LeeVerley2007, GorbanKegl2008}. Various methods of non-manifold low-dimensional data approximation have been also developed \cite{GorbanZin2010Neural,ZhangPedrycz2013} and used for solving applied problems \cite{GolovenBac2020,ChenPinello2019}. In addition to the injective data approximation methods, when the approximant is embedded in the dataspace, a large family of projective methods has been developed in which only the projection function onto the base space is learnt (see a        brief review with a comparison of modern injective and projective methods in \cite{BacZinovyev2020}).

The development of mathematical models of complex networks requires effective model reduction technology \cite{GorbanKaz2006}. Many methods were created in the field of chemical dynamics \cite{Angeli2009} at the beginning of the twentieth century and then received mathematical justification. Here we should mention the 1956 Nobel Prizes awarded to Semenov and Hinshelwood, for research based on effective dimensionality reduction by the method of quasi-stationary concentrations (for more information, we refer to the reviews of \cite{GorbanYablons2015, Gorban2018} and the textbook \cite{MarinYablon2019}). The very popular machine learning algorithms for nonlinear dimensionality reduction and manifold learning, autoencoders, were invented first for chemical applications \cite{Kramer1991,Kramer1992}.

Two very general mathematical ideas are used in dimensionality reduction for large dynamical systems: \emph{invariant manifold} and \emph{Lyapunov function}. The motion reduced  to an invariant manifold never leaves it. If this manifold is, at the same time, a slow manifold, then the game is over and the reduced description is ready \cite{GorbanKarlin2005, GorbanKarlin2014}. 
Methods of invariant manifold aim to find a good approximation to the unknown slow invariant manifold.

Lyapunov functions monotonically decrease (non increase) along the systems trajectories. They appear in various forms: entropy or free energy (for physical and chemical systems \cite{Jaynes1957, GorbanKarlin2005, Grmela2013, PavelkaGrmela2018}) and informational entropy or various divergences for machine learning \cite{Principe2010, Filippi2010, Nowozin2016, Semenov2019, Ji2020}, generalized fitness (for evolutionary modelling \cite{Gorban2007, KarevaKarev2019,Morozov2019}), loss function \cite{Li..Goldstein2018} (for machine learning), or just general Lyapunov functions for many general problems of control, machine learning and adaptation in dynamical systems \cite{Tyukin2011}.  

Information about the Lyapunov function for large dynamical systems may be much more reliable than the information about dynamic (vector field) itself. In physical and chemical applications we can operate by entropy, free energy and other thermodynamic potentials even when the detailed kinetics is not well-known. Similarly, in machine learning we can define the loss functions before detailed specification of the learning algorithm. In adaptive control and observation problem, we also define the quality functionals before construction of observers and adaptation algorithm.

Definition of the Lyapunov function for a reduced system is very simple: we can just restrict the Lyapunov function of the initial system on the manifold of reduced description. But projection of the original dynamics is not so obvious. Of course, if the manifold of the reduced description is invariant with respect to original dynamics, then the original vector field is tangent to this manifold and belongs to it tangent space. In this case, no projection problem arise. But in practice the manifolds of reduced description are newer invariant beyond oversimplified toy examples. Therefore, we come to the projection problem: how should we project the original vector field onto the tangent space of an ``ansatz manifold'' in order to guarantee preservation of dissipativity: if the time derivative of the Lyapunov function according to the original vector field is non-positive, then the time derivative of the restriction of this Lyapunov function according to the projected vector field should be also non-positive.    

This projection problem becomes interesting and non-trivial if we consider it in the universal settings:
\begin{itemize}
\item Let the vector field be a priory unknown; the only available a priory information is: it has the given Lyapunov function;
\item Let the ansatz manifold be apriori unknown; the only available information is some regularity of intersections with the Lyapunov function levels (the detail follow);
\item Find a projector that transforms any a priory unknown vector field with a given Lyapunov function into a vector field on a regular ansatz manifold equipped with a Lyapunov function, which is a simple restriction of the original Lyapunov function on this ansatz manifold.
 \end{itemize}

Such a universal approach arose from the most universal physical discipline -- thermodynamics. The main example was the theory of strong shock waves. The original (non-reduced) system was the general Boltzmann equation. In the middle of XXth century, Tamm and Mott-Smith proposed a very simple and useful ansatz for analysis of strong shock waves. The distribution function of particle speeds was presented as a linear combination of Maxwellian distribution  at the ``beginning'' ($t\to -\infty$)  and at the ``end'' ($t\to +\infty$). The new variable of  reduced description is just the coefficient of this combination. Then the Boltzmann equation was projected onto this ansatz. The result was useful and sufficiently simple for analytic and numerical analysis (see, for example, the book \cite{Cercignani}). Nevertheless, the projected equation may violate the Second law of thermodynamics: in some situations, entropy decreases \cite{Naka}. In 1977, M. Lampis proposed a projector for the Tamm--Mott-Smith ansatz that guaranteed preservation of dissipativity \cite{Lampis1977}. Later on, Gorban and Karlin proposed a general thermodynamic theory for the  Tamm--Mott-Smith ansatz and similar approximations \cite{GorbanKarlin1992}. In their work, the quasi-one dimensional nature of the Tamm--Mott-Smith ansatz was used. Some generalization we proposed in 1994 for solution the problem of instabilities of post-Navier--Stokes approximations \cite{GorbanKarlin1994}. In 2004, they published the general construction of thermodynamic projector for more general ansatz manifolds \cite{GorbanKarlin2004} and applied it to the Hilbert 6th problem for analysis of hydrodynamic ansatz manifolds in kinetic theory \cite{GorbanKarlin2014}.

There are three surprises:
\begin{enumerate}
\item The universal projector that preserves dissipativity do exist;
\item This projector is unique: for all other projectors there are counterexamples with non-monotone dynamic of Lyapunov function along some trajectories of reduced systems;
\item This projector preserves not only the sign of the Lyapunov function time derivative (that is the initial requirement) but also the exact value of this derivative. This is a rare case of a ``free lunch''.
\end{enumerate}

In our work, we analyse the mathematical problem of preservation of dissipativity in general dynamical models with a given Lyapunov function, give the explicit formula for the universal projector, and prove  that it is unique. In Section \ref{Formal} we give the rigorous problem statement. In Section \ref{Projector} the main theorem about universal projector are proven. In Section \ref{monoton} we consider the model reduction methods beyond manifolds. For this purpose, we introduce  the {\it monotone trees} and  demonstrate how the results of the previous section can be used for unambiguous projecting of dissipative dynamics onto such trees.
In Section Conclusion and outlook we discuss the results and the possible directions of future work.

\section{Notations and formal problem statement \label{Formal}}

Let $U$ be a open convex subset of $n$-dimensional Euclidean space $\mathbb{R}^n$.
The a priori given Lyapunov function $H$ is given in $U$. Assume that:
\begin{itemize}
\item $H$ is twice-differentiable.
\item $H$ is strongly convex in $U$, that is,  the second differential of $H$, $D^2_x(H)$  is uniformly positive definite in $U$: all its eigenvalues $\lambda$ are positive and $\lambda > \alpha >0$ for some $\alpha$ for all $x\in U$.
\end{itemize}

In a coordinate system $x_1, \ldots, x_n$, the second differential is presented by the Hessian matrix 
$$D^2_x(H)=\frac{\partial^2 H}{\partial x_i \partial x_j}.$$ 

A $m$-dimensional ansatz manifold $M$ is a twice-differentiable immersed $m$-dimensional submanifold in $U$.   

Because most of our problems are local, it is convenient to work with a coordinate chart of $M$. 
It is an injective immersion $F:B_m \to U$, where  $B_m$ is an open convex subset of $m$-dimensional Euclidean space $\mathbb{R}^m$. This means that $F$ is injective and the differential of $F$ at each point $p\in B_m$, $D_p(F)$,  is an injective linear map from  $\mathbb{R}^m$ to $\mathbb{R}^n$. 

For given  coordinate  systems in spaces $\mathbb{R}^m$ ($p_1, \ldots, p_n$) and $\mathbb{R}^n$ ($x_1, \ldots, x_n$), the map $F$ is given by $n$ functions on $B_m$: $x_i=f_i (p_1,\ldots , p_m)$ ($n=1, \ldots, n$). 

The differential $D_p(F)$ is presented by the matrix of partial derivatives $\left.\frac{\partial f_i}{\partial p_j}\right|_p$. 

$F$ is immersion if and only if rank of $D_p$ is constant and equal to $m$: ${\rm rank}(D_p(F))=m$ for all $p\in B_m$.

Of course, we can always assume that $B_m$ is an open $m$-dimensional ball. Nevertheless, a bit more general definition with open convex sets may be useful for some examples.  

There is no simple general local criterion for the injectivity of $F$, but for some practical examples it can often be proved without difficulty. (For example, for the Tamm--Mott-Smith anzatz discussed above: If for two pairs of Maxwellians two linear combinations  of them with non-zero coefficients coincide, then these pairs coincide too because of linear independence of different Maxwellians. This means injectivity of the Tamm--Mott-Smith ansatz.) 

The tangent space to $M$ at a point $x=F(p)$ is the image of the differential $D_p(F)$: $T|_{F(p)}F(B_m)={\rm im} D_p(F)$. In the coordinate form, it can be presented as a linear span of linear independent vectors of partial derivatives (columns of the matrix $\left.\frac{\partial f_i}{\partial p_j}\right|_p$): 
$$T_{F(p)}={\rm span} \left\{\left(\frac{\partial x_1}{\partial p_i},\frac{\partial x_2}{\partial p_i}\ldots \frac{\partial x_n}{\partial p_i}\right), \quad i=1,\ldots, m\right\}\ .$$

The tangent space of $U$ at any point may be considered just as $\mathbb{R}^n$. The first differential of $H$ in $U$, $D_x H$, is a linear functional: for each $y\in \mathbb{R}^n$
$$D_x (H) (y)=\sum_i \frac{\partial H}{\partial x_i} y_i .$$

Restriction of $H$ on $M$ at point $x=F(p)$ is just the value of $H$: $H|_M(F(p))= H(F(p))$. Differential of restriction of $H$ on $M$ at point $x=F(p)$ is restriction of the functional $D_x (H)$ on $T_{F(p)}$. Formally, if a vector $y\in T_{F(p)}$, this means  
$$y=\sum_{i=1}^m a_i \left(\frac{\partial x_1}{\partial p_i},\frac{\partial x_2}{\partial p_i}\ldots \frac{\partial x_n}{\partial p_i}\right)$$
for some numbers $a_1,\ldots , a_m$. For this $y$,
$$D_x (H_M) (y)=\sum_{i=1}^m a_i \sum_{j=1}^n \frac{\partial H}{\partial x_j} \frac{\partial x_j}{\partial p_i}.$$

If we transform the basis in $T|_{F(p)}$ to $\{e_1, \ldots , e_m\}$ then for calculation of $D_x (H_L) (y)$ for $y=\sum b_i e_i$ we should use representation of $e_i$ in the old basis to  find  $D_x (H) (e_i)$ and then calculate $D_x (H_M) (y)=\sum b_i D_x (H) (e_i)$.

The second differential of $H$ induces  a Riemannian metric on $U$ for which the Hessian of $H$ is the metric tensor. This is the so-called Shahshahani metric \cite{Shahshahani1972}. The inner product  associated with this metric in the tangent space at a point $x\in U$ is defined as
$$\langle y | z \rangle_x = (y,D^2_x(H) z),$$
or in the coordinate form
$$\langle y | z \rangle_x =\sum_{i,j=1}^n y_i \left.\frac{\partial^2 H}{\partial x_i \partial x_j}\right|_x z_j$$

Of course, the Shachshahani metric generated by the second differential of $H|_M$ is also defined on the submanifold $M$. We use notation $\mathcal{P}^{\perp}_x$ for the orthogonal projectors of $\mathbb{R}^n$ onto $T_x(M)$ in the Shahshahani metrics $\langle \cdot |\cdot \rangle_x$. These projectors do not solve the main problem but are used in the solution.

Consider a vector field $W(x)$ in $U$. If it has Lyapunov function $H$ then the derivative of $H$ along $W(x)$ is non-positive. In our notations, it means that 
\begin{equation}\label{Diss1}
D_x (H) (W(x))\leq 0 \mbox{ for all } x \in U ,
\end{equation}
or in the coordinate form
\begin{equation}\label{Diss2}
\sum_i w_i(x) \left.\frac{\partial H}{\partial x_i}\right|_x \leq 0
\end{equation}
if $w_i(x)$ are the components of the vector field $W(x)$. 

Consider a vector field $Q(x)\in T_{F(p)}$ defined at points $x=F(p)$ of  the submanifold $M$. Definition of the derivatives of $H_M$ along this field is the same as (\ref{Diss1}), (\ref{Diss2}). Taking into account that $x\in M$ and $H_M=H$ on $M$, we write    
$$D_x (H_M) (Q(x))=D_x(H)(Q(x)).$$

The problem is: For $x\in M$ find projectors $\mathcal{P}_x: \mathbb{R}^n \to T_x M$ such that for any dissipative vector field $W(x)$ (\ref{Diss1}), (\ref{Diss2}) with Lyapunov function $H$ the projected vector field $Q(x)$ is also dissipative with the restriction $H$ on $T_x(M)$, $H_M$. 

The projectors $\mathcal{P}_x$ depend on point $x\in M$. They also depend on the tangent space $ T_x M$.  We require also that the unknown yet field of projectors should be smooth. Usually, it is clear, for which manifold $M$ and tangent space $ T_x M$ a projector is $\mathcal{P}_x$ build, and we omit the cumbersome notation, but if it is necessary, we can use $\mathcal{P}_{x,M}$ for projector of $\mathbb{R}^n$ onto
$ T_x M$.

\section{Universal projector for submanifolds \label{Projector}}

\subsection{Linear submanifolds, linear vector fields, and quadratic Lyapunov functions}

Linear vector fields with quadratic Lyapunov functions and linear ansatz submanifold provide us with the simplified version of the problem of dissipative projection and, at the same time, demonstrate how to use the Shahshahani metrics for the construction of the projector.

Let $U=\mathbb{R}^n$ and $H$ be the positive definite quadratic form $H(x)=
(x, G x)$, where $G$ is a positively definite symmetric operator and $(\cdot , \cdot )$ is the standard inner product. We select the auxiliary (Shahshahani) inner product $\langle y|z \rangle = (y, G z) = (G y, z)$.

Consider a linear vector field $Q(x)=Qx$, where $Q$ is a linear operator. It is well-known that this vector field has Lyapunov functions $H$ if and only if the time derivative of $H$ along $Q$ is non-positive: for all $x\in \mathbb{R}^n$
\begin{equation}
\begin{split}
\frac{d}{dt}(x,Gx)&=(\dot{x} , Gx)+ (x, G \dot{x})\\&= (Qx,Gx)+(x, GQx)=(x,(Q^TG+GQ)x)\leq 0. 
\end{split}
\end{equation}
This means that the symmetric operator $Q^TG+GQ$ is negative semi-definite. The equivalent inequality is 
\begin{equation}\label{DissLin}
\langle Qx|x \rangle \leq 0 \mbox{ for all } x\in \mathbb{R}^n .
\end{equation}
A linear ansatz manifold is just a linear subspace in $\mathbb{R}^n$. In this case we can use for the tangent space $T(M)$ the same notation $M$.  A projector $\mathcal{P}: \mathbb{R}^n\to M$ is a surjective linear operator (${\rm im} \mathcal{P}=M$) with the property $\mathcal{P}^2=\mathcal{P}$, or, which is the same, such a linear operator $\mathcal{P}: \mathbb{R}^n\to M$ that its restriction on $M$ is the unite operator: $\mathcal{P}|_M={\rm id}_M$.  

For a given linear anzatz manifold $M$, projector $\mathcal{P}: \mathbb{R}^n\to M$ and a linear vector field $Q(x)=Qx$ on $\mathbb{R}^n$, the projected vector field on $M$ is $\mathcal{P}Qx$, $x\in M$. It is convenient to represent this projection in a more symmetric form: $\mathcal{P}Qx=\mathcal{P}Q\mathcal{P}x$ for  $x\in M$.

\begin{theorem}\label{Theor:Linear}
Proector $\mathcal{P}: \mathbb{R}^n\to M$ projects any linear vector field $Qx$ with quadratic Lyapunov function $H(x)= \langle x | x\rangle$ into a linear vector field  on a linear subspace $M$ with the same quadratic Lyapunov function $H_M(x)= \langle x | x\rangle$ (for $x\in M$) if and only if $\mathcal{P}$ is an orthogonal projector $\mathcal{P}^{\perp}$ in the Shahshahani inner product $\langle y | z\rangle$.
\end{theorem}
\begin{proof}
Assume that the projector $\mathcal{P}$ is an orthogonal projector in the Shahshahani metric (image $\mathcal{P}$ is orthogonal to kernel $\mathcal{P}$). An equivalent condition is that projector $\mathcal{P}$ is a self-adjoint operator in this inner product: $\langle \mathcal{P} y | z\rangle=\langle y | \mathcal{P} z\rangle$ for all $y,z \in  \mathbb{R}^n$. Let an operator $Q$ satisfy the dissipativity condition (\ref{DissLin}). Then for its projection onto $M$ we get $\langle \mathcal{P}Q\mathcal{P}x|x \rangle=\langle Q\mathcal{P}x|\mathcal{P}x \rangle$ because $ \mathcal{P}$ is self-adjoint. This form is negative semi-definite because  $\langle Qy|y \rangle \leq 0$ for all $y\in \mathbb{R}^n$.

Assume now that a projector $\mathcal{P}:\mathbb{R}^n \to M$ is not an orthogonal projector in the Shahshahani metric. Then the orthogonal complement to $M= {\rm im}\mathcal{P}$ does not coincide with the kernel of $\mathcal{P}$ and there exists such a vector $y\in  \mathbb{R}^n$ that $\langle y | z\rangle = 0$ for all $z\in M$ but $Py\neq 0$. 
Consider an one-parametric family $ A_a$ of auxiliary rank one operators: for any $x \in  \mathbb{R}^n$
$$ A_a x=- (\mathcal{P}y-ay)\langle \mathcal{P}y-ay | x \rangle .$$
These operators are negatively semi-definite: $\langle x| A_a x \rangle =-\langle x | \mathcal{P}y-ay \rangle^2\leq 0$. Calculate the projection of these operators onto $M$. Note that for $x\in M$
$A_a  x= -(\mathcal{P}y-ay) \langle \mathcal{P}y-ay | x \rangle=-(\mathcal{P}y-ay) \langle \mathcal{P}y| x \rangle$ because $\langle y | x \rangle=0$. Projection of these vectors on $M$ gives $\mathcal{P}A_a  x =-(1-a)\mathcal{P}y \langle \mathcal{P}y| x \rangle$. For $a>1$ these operators are not negatively semi-definite: $\langle x |\mathcal{P}A_a  x\rangle=(a-1)\langle \mathcal{P}y | x\rangle^2$.
\end{proof}

Thus, for linear submanifolds (subspaces) and quadratic Lyapunov functions preservation of dissipativity in reduced systems requires orthogonal projection in the Shahshahani metrics. The next step is the analysis of the problem near the equilibrium that is the global minimizer of $H$ in $U$.

\subsection{Extension of dissipative vector fields}

Transition from local analysis (in a vicinity of a point or a compact) to global statements is a very common problem. In general topology, continuation of a continuous function from a closed subset to the entire space is given by the Tietze--Urysohn--Brouwer extension theorem. For differentiable functions,  the Whitney extension theorem states that it is possible to extend a differentiable function with given derivatives from a closed subset in $\mathbb{R}^n$ to the whole space \cite{Whitney1944, Malgrange1967}. 

We have to extend a locally given twice differentiable dissipative vector field to such a field in the entire $U$. The idea is very simple: we will use the classical theorems about extension of differentiable functions and ``glue'' the locally given dissipative field $Q(x)$ with the globally given field ${\rm grad} H$.   

\begin{theorem}\label{TheorExtens}
Let $K$ be compactly contained in $U$: $K \Subset U$, and its $\varepsilon$ neighbourhood $K_{\varepsilon}=K+\varepsilon B^n$ (where $B^n$ is the unite ball in $\mathbb{R}^n$) is also compactly contained in $U$. If $Q$ is a dissipative vector field in $K_{\varepsilon}$ then there exist a dissipative vector field $Q^{\rm ext}$ in $U$ such that $Q(x)=Q^{\rm ext}(x)$ for $x\in K$. 
\end{theorem}
\begin{proof}
According to the extension theorems for differentiable functions, there exists such an infinitely smooth function $f_{K, \varepsilon}$ on $U$ that $f_{K, \varepsilon}(x)\geq 0$ in $U$, $f_{K, \varepsilon}(x)=1$ for $x\in K$ and $f_{K, \varepsilon}(x)=0$ for $x\notin K_{\varepsilon}$. Define $Q^*(x)= f_{K, \varepsilon}(x)Q(x)$ for $x\in K_{\varepsilon}$ and $Q^*(x)=0$ for $x\notin K_{\varepsilon}$. For any globally dissipative vector field $W$ in $U$, the linear combination $Q^{\rm ext}(x)=Q^*(x)+ (1-f_{K, \varepsilon}(x))W$ is a globally dissipative vector field in $U$. If $x\in K$ then $Q^{\rm ext}(x)=Q(x)$, and if $x\notin K_{\varepsilon}$ then $Q^{\rm ext}(x)=W(x)$.
\end{proof}

This theorem allows us to discuss locally dissipative vector fields and then extend the results to the entire $U$. In particular, it gives the possibility to extend Theorem \ref{Theor:Linear} from exactly linear submanifolds, linear vector fields, and quadratic Lyapunov functions  to vicinities of equilibria with  almost linear submanifolds, almost linear vector fields, and almost quadratic Lyapunov functions.

A special case is the extension of a dissipative vector field from a submanifold in $U$ to entire $U$. Further, we use the extension of locally given vector fields in a vicinity of a point  $x\in M$. Let $Q$ be a smooth dissipative vector field on $M$ in a $\varepsilon$ vicinity of $x$  that is $D_y (H_M)(Q)\leq 0$ for $y\in M$ and $\|x-y\|<\varepsilon$ ($\varepsilon>0$).

\begin{proposition}\label{ProposExtend}
\begin{enumerate}
\item If $D_x(H_M)(Q(x))< 0$ then there exist a dissipative vector field $Q^{\rm ext}$ in $U$ such that $Q(y)=Q^{\rm ext}(y)$ for $y\in M$ and $\|x-y\|< \delta$ for some $\delta>0$.
\item If  $Q(x)= 0$ but  $D_y(H_M) (Q(y))<-\lambda \|x-y\|$ for $y\in M$, $\|x-y\|< \varepsilon$ and some  $\lambda>0$ then there exist a dissipative vector field $Q^{\rm ext}$ in $U$ such that $Q(y)=Q^{\rm ext}(y)$ for $y\in M$ and $\|x-y\|< \delta$ for some $\delta>0$.
\end{enumerate}
\end{proposition}
\begin{proof}
Locally, in a small vicinity of $x$ in $U$, we can introduce such a coordinate systems $y=(y_1, \ldots , y_n)$ that $M$ is given by  the equations $y_{m+1}= \ldots = y_n=0$, where $m=\dim M$ ($\|y- x\|< \kappa$, $y \in U$ for some $\kappa> 0$). Locally, for $\|y- x\|< \kappa$, $(y_1, \ldots , y_m)$ is a coordinate system on $M$. Let us use this coordinate system.

Let us prove statement 1. If $Q(x)\neq 0$ then $x$ is not a critical point of $H_M$ on $M$. Indeed, if $x$ is a critical point of $H_M$ on $M$ then,  because strong convexity of $H$, $x$ is a point of local minimum of $H_M$ on $M$, and  $D_{x+\gamma z}(H_M))(Q(x+\gamma z))>0$ and $D_{x-\gamma z}(H_M))(Q(x-\gamma z))<0$ for some $z=(z_1, \ldots, z_m, ),\ldots,0$ and sufficiently small $\gamma>0$.

Let us use the standard theorem about rectification of vector fields (see, for example, \cite{Arnold1992,Arnold2012}. In a small vicinity of $x$ we can transform the coordinate system $(y_1, \ldots, y_m)$ on $M$ in such a way that $Q(y)=(1,0,\ldots, 0)$ ($m-1$ zeros, $y\in M$, $\|x-y\|< \nu $ for some $\nu<0$). This transformation gives, at the same time, an extension of $Q(y)$ onto a vicinity of $x$ in $U$: $Q(y)=(1,0,\ldots, 0)$ ($n-1$ zeros, $y\in U$, $\|x-y\|< \varsigma $ for some $\varsigma>0$).

Thus, we found an extension of $Q$ onto a vicinity of $x$ in $U$. After that, we can apply Theorem \ref{TheorExtens}. The first statement is proven.

Let us prove the statement 2.  If $Q(x)=0$ then the used above simple version of rectification theorem is not applicable. For any given sufficiently small $(y_{m+1}, \ldots , y_n)$ the set $\{ (y_1, \ldots , y_m, y_{m+1}, \ldots , y_n)\}$ with variable $y_1, \ldots , y_m$ in a vicinity of $x$ is a $m$-dimensional immersed submanifold. In the selected coordinate system, this submanifold is just a shift of a vicinity of $x$ in $M$ by vector $(0, \ldots, 0, y_{m+1}, \ldots , y_n)$. We use for the result of this shift the notation $M(y_{m+1}, \ldots , y_n)$. In this notation, a sufficiently small vicinity of $x$ in $M$ is $M(0, \ldots , 0)$ ($n-m$ zeros). For convenience, put $x$ in the origin: $x=(0,\ldots, 0)$ ($n$ zeros). 

According to the assumptions,  $Q(x)= 0$, but  $D_y(H_M) (Q(y))<-\lambda \|x-y\|$ for $y\in M$, $\|x-y\|< \varepsilon$ and some  $\lambda>0$. Therefore $x$ is a critical point of $H_M$ on $M$. Due to the  strong convexity of $H$ on $U$, we can use the implicit function theorem and find the critical point $x^*(y_{m+1}, \ldots , y_n)$ of $H_{M(y_{m+1}, \ldots , y_n)}$ on $M(y_{m+1}, \ldots , y_n)$, such that $x^*(y_{m+1}, \ldots , y_n)$ is a smooth function and $x^*(0, \ldots , 0)=x$.  

Let us define the extended vector field in the vicinity $x$ in $U$ (extension from dimension $m$ to dimension $n$):
$$Q_n(y)=Q(y-x^*(y_{m+1}, \ldots , y_n)).$$
For this vector field, 
$D_y(H) (Q(y))<-\lambda' \|x-y\|$ for $\|x-y\|< \varepsilon'$ and some  $\lambda'>0$, $\lambda'<\lambda$, $\varepsilon'<\varepsilon$. For the further extension to the entire $U$ we use Theorem \ref{TheorExtens}. The second statement is proven.
\end{proof}

\subsection{Near equilibrium \label{SSecEqui}}

Assume that $x^{\rm eq}\in U$ is the global minimiser of $H$ in $U$. Each dissipative vector fields is zero at $x^{\rm eq}\in U$. Therefore, there exist no standard ``pointwise'' restriction on $\mathcal{P}_{x^{\rm eq}}$: each projector projects zero into zero.  Nevertheless, the dissipativity in the neighbourhood of $x^{\rm eq}\in U$ defines $\mathcal{P}_{x^{\rm eq}}$ unambigously.

To simplify the notation, let us choose the origin at the equilibrium point: $x^{\rm eq}=0$. notice that $H(x)=\frac{1}{2}\langle x | x \rangle_0 + o(\|x\|^2)$. Near the equilibrium,  the linear functional $D_x (H)(z)$ can be approximated as $\langle x | z \rangle_0$: 
$$|D_x (H)(z)-\langle x | z\rangle_0| < o(\|x\|) \|z\|.$$

 Following Theorem \ref{Theor:Linear}, we can guess that $\mathcal{P}_{x^{\rm eq}}=\mathcal{P}_{x^{\rm eq}}^{\perp}$. To prove this guess, we have to consider projections in a vicinity of $x^{\rm eq}$ because exactly at equilibrium, all dissipative vector fields vanish. Near $x^{\rm eq}$ projectors that preserve dissipativity are close to $\mathcal{P}_{x^{\rm eq}}$. The following theorem formalizes our guess.

\begin{theorem}\label{Theo:NearEq}
For preservation of dissipativity by differentiable projector field $\mathcal{P}_x$, it is necessary that
\begin{equation}\label{ProjEq}
\mathcal{P}_x=\mathcal{P}_x^{\perp}+ O(\|x-x^{\rm eq}\|).
\end{equation}
\end{theorem}
\begin{proof}
The equivalent formulation of this theorem is: $\mathcal{P}_{x^{\rm eq}}=\mathcal{P}_{x^{\rm eq}}^{\perp}$. At equilibrium, the projector is just an orthogonal projector in the Shahshahani metrics. Recall that the Hessian of $H$ is strictly negative definite at any point $x\in U$.  

Let us construct a proof by contradiction. Assume that $M$ is a manifold of reduced description and $0\in M$. Let $\mathcal{P}_{0,M}$ be not an orthogonal projector in the Shahshahani inner product $\mathcal{P}_{0,M}=\tilde{\mathcal{P}}\neq  \mathcal{P}_{0}^{\perp}$. The images of these projectors with  $T_{0}(M)$. Therefore, they should have different kernels: there exists such a vector $y\in  \mathbb{R}^n$ that $\langle y | z\rangle_{0} = 0$ for all $z\in T_{0}(M)$ but $\tilde{\mathcal{P}}y\neq 0$, exactly as in the proof of Theorem \ref{Theor:Linear}. 

Consider a one-parametric family of auxiliary vector fields in the vicinity of the equilibrium:

$$B_a(x)= -(\mathcal{P}_x y-ay)D_x(H) (\mathcal{P}_x y-ay).$$ 

Here we assume that projectors $\mathcal{P}_x$ project dissipative vector fields on $\mathbb{R}^n$ into a dissipative vector field on $T_x(M)$ (in our simplified notation, $T_x(M)=M$. According to the problem statement, this field of projectors should be continuous and $\|\mathcal{P}_x-\mathcal{P}_0\|=O(\|x\|)$.

These vector fields are dissipative: 
$$D_x(H)(B_a(x))=-(D_x(H) (\mathcal{P}_x y-ay))^2\leq 0.$$

Let us demonstrate that projection of $B_a(x)$ on $M$ is not dissipative. For $x\in M$, $\mathcal{P}_x B_a(x)=(a-1)\mathcal{P}_x y D_x(H) (\mathcal{P}_x y-ay)$ and $D_x H (\mathcal{P}_y-ay)= \langle x | \mathcal{P}_xy \rangle_0 +o(\|x\|)$.

Therefore, for $x\in M$, the time derivative of $H$ according to the  projected vector field is

$$D_x(H)(\mathcal{P}_x B_a(x))=(a-1)\langle x| \mathcal{P}_x y \rangle_0^2 + o(\|x\|^2).$$

This expression is positive for $a>1$ and some sufficiently small $x$, for example, for 
$x=\varepsilon \mathcal{P}_x y$ and sufficiently small $\varepsilon>0$. In these cases, the projected vector field is not dissipative despite of dissipativity of the original vector field in $U$.

According to Theorem \ref{TheorExtens}, vector fields dissipative in some vicinity of equilibrium can be extended to the entire $U$. Therefore, we proved that $\mathcal{P}_{x^{\rm eq}}=\mathcal{P}_{x^{\rm eq}}^{\perp}$ and $\mathcal{P}_x=\mathcal{P}_x^{\perp}+ O(\|x-x^{\rm eq}\|)$ for projectors that preserve dissipativity. \end{proof}

This theorem  tells us that the difference between the projectors that preserve dissipativity of all dissipative vector fields and the orthogonal projector (in the Shahshahani metrics) grows out of equilibrium not faster than $O(\|x-x^{\rm eq}\|)$. 

For linear submanifolds $L\subset \mathbb{R}^n$, theorem \ref{Theo:NearEq} will be applied also to dissipative vector fields on  $L\cap U$ of $\mathbb{R}^n$ near minima of $H$ on $L \cap U$ without changes.

\subsection{One-dimensional submanifolds \label{SSeq1D}}

The orthogonal projectors in the Shahshahani metrics solve sometimes the problem of projection with preservation of dissipativity, but already simple example show that this is not the general solution of the problem. Already projections onto one-dimensional manifolds (curves) require different approach. Reduction to essentially one-dimensional cases is not an exotic and oversimplified example. The famous Tamm -- Mott-Smith approximation leads to this problem.

Return to the general problem statement (Sec.~\ref{Formal}). Let the manifold of reduced description $M$ be one-dimensional (a curve). Its tangent space at point $x\in M$ is a one-dimensional subspace $T_x(M) \subset \mathbb{R}^n$. Select a normalised vector $e_x\in T_x(M)$, $\langle e_x| e_x \rangle= 1$. A projector $\mathcal{P}$ can be presented by two subspace: ${\rm ker} \mathcal{P}$ and ${\rm im} \mathcal{P}$ under conditions that 
\begin{equation}\label{ProjectorCond}
{\rm ker} \mathcal{P} \cap {\rm im} \mathcal{P}= \{0\} \mbox{ and }\dim {\rm ker} \mathcal{P}+ \dim {\rm im} \mathcal{P}=n
\end{equation}

The image of $\mathcal{P}_x$ is, obviously, $T_x(M)$. In this subsection, this is the straight line: $T_x(M)=\{\alpha e_x | \alpha \in R\}$, where $e_x$ is a basis vector of $T_x(M)$. The following theorem gives us the kernel of $P_x$.
\begin{theorem}\label{Theor1Dpro}
Assume that the manifold of reduced description $M$ is one-dimensional and  $D_x (H)(e_x)\neq 0$. Then for the preservation of dissipativity in the projection for all dissipative vector fields  it is necessary and sufficient that ${\rm im} \mathcal{P}_x=T_x M$ and ${\rm ker} \mathcal{P}_x = \ker D_x (H)$.
\end{theorem}
\begin{proof}
Let ${\rm im} \mathcal{P}_x=T_x (M)$ and ${\rm ker} \mathcal{P}_x = \ker D_x (H)$. Any linear operator $\mathcal{P}$ with given one-dimensional image $M_x$ and kernel $\ker D_x (H)$ can be presented as 
$$\mathcal{P}Q=ae_x D_x (H)(Q)$$
with a real parameter $a$. For the projector $\mathcal{P}_x$ this parameter can be defined using the projector property $\mathcal{P}_x^2= \mathcal{P}_x$: 
$$a^2 e_x (D_x (H)(e_x))D_x (H)(Q)=ae_x D_x (H)(Q).$$
Therefore, $a=(D_x (H)(e_x))^{-1}$ and
\begin{equation}\label{1Dprojector}
\mathcal{P}_xQ=e_x \frac{D_x (H)(Q)}{D_x (H)(e_x)}.
\end{equation}
This is the projector with the given kernel and image.

The time derivative of $H_M$ along the projected vector field $\mathcal{P}_xQ$ (\ref{1Dprojector}) is
$$D_x (H)(\mathcal{P}_xQ)=D_x (H)(e_x)\frac{D_x (H)(Q)}{D_x (H)(e_x)}=D_x (H)(Q).$$
We can see that the time derivative of $H$ does not change after dimensionality reduction with the projector $\mathcal{P}_xQ$ (\ref{1Dprojector}). This is even more then we are looking for: Preservation of the dissipativity assumes preservation of the sign of $dH/dt$, and here we see preservation of the value.

All projectors onto $T_x(M)$ have the same image, $T_x(M)$, but the kernels might be different under conditions (\ref{ProjectorCond}). Let us demonstrate that all kernels of projector other than the null-space of the differential of $H$, $\ker D_x (H)$, lead to violation of dissipativity for some vector fields $Q$ in projection onto $T_x(M)$. 

Consider a projector $\tilde{\mathcal{P}}_x:\mathbb{R}^m\to T_x(M)$ that is different from $\mathcal{P}_x$ given by   formula (\ref{1Dprojector}). The general form of projector on a one-dimensional subspace is $\tilde{\mathcal{P}}_x (Q)=e_x l(Q)$, where $l$ is a linear functional, and $l(e_x) =1$ because projector condition $\mathcal{P}^2=1$. According to projecting of a vector field $Q$ by the projector $\tilde{\mathcal{P}}_x$, the time derivative of $H_M$ is $D_x(H)(e_x) l(Q)$. The preservation of dissipativity requires that this time derivative is non-positive when the original dissipations, $D_x(Q)$ is non-positive. This means that the half-spaces in $\mathbb{R}^n$ given by the inequalities $\{y\in \mathbb{R}^n|l(y)\leq 0\}$ and 
$\{y\in \mathbb{R}^n|D_x(H)(y)\leq 0\}$ coincide. For two linear functionals it is possible if and only if they are proportional with a positive factor. Hence, the preservation of dissipativity requires $l=aD_x(H)$, $a>0$, and the projector $\tilde{\mathcal{P}}$ coincides with $\mathcal{P}_x$ (\ref{1Dprojector}).

Thus, the requirement of preservation of dissipativity unambigously defines projector $\mathcal{P}_x$ (\ref{1Dprojector}). Moreover, the preservation of the sign of dissipation implies the preservation of the value of $dH/dt$ in the reduces system: $dH/dt=dH_M/dt$
\end{proof}

Informally, the proven construction of projector (\ref{1Dprojector}) for one-dimensional manifolds $M$,  can be characterised choosing $H_M$ as the internal coordinate $h$ on $M$ and defining the projector $\mathcal{P}_x$ as follows: for a vector field $W(x)$
$$\frac{d h}{dt} = \frac{d H}{dt}= D_x (U)(W(x)).$$
Here on the left we see movement on $M$ due to the projected vector field (a reduced system), and on the right we see the time derivative of $H$ due to the vector field $W$ in the original system.

Formula for one-dimensional projections (\ref{1Dprojector}) is important for applications, therefore we reproduce it in coordinate form. 

Let a one-dimensional manifold (a curve) be given by $n$ functions of one variable: $x=(f_1(p),\ldots , f_n(p))$.  Here $p$ is a real variable (parameter), and $f_i$ are the twice differentiable functions. 
\begin{itemize}
\item For a given $x=(f_1(p),\ldots , f_n(p))$, $e_x$ is a vector of derivatives, $e_x=(df_1(p)/dp, \ldots df_n(p)/dp)$, and some of these derivatives are non-zero. 
\item The linear functional $D_x H$ is a linear function with coefficients $\partial H/\partial x_i$.
\item The vector field $Q$ is $Q=(q_1, \ldots ,q_n)$.
\item The value of the functional, $D_x (H)(Q)$ is 
$$D_x (H)(Q)=\sum_{i=1}^n \frac{\partial H}{\partial x_i} q_i.$$
\item The value of the functional $D_x (H)(e_x)$ is
$$D_x (H)(e_x)=\sum_{i=1}^n \frac{\partial H}{\partial x_i} \frac{df_i}{dp}.$$
\item The final formula (\ref{1Dprojector}) for the projector $\mathcal{P}_x$ looks in coordinates as follows:
\begin{equation*}
(\mathcal{P}_xQ)_i=\frac{d f_i}{dp}\frac{\sum_{j=1}^n \frac{\partial H}{\partial x_j} q_j}{\sum_{j=1}^n \frac{\partial H}{\partial x_j} \frac{df_j}{dp}}.
\end{equation*}
\end{itemize}
In particular, formula (\ref{1Dprojector}) and Theorem \ref{Theor1Dpro} explain the mathematical foundations of the results of Lampis \cite{Lampis1977} and Gorban and Karlin \cite{GorbanKarlin1992} about the projector of the Boltzmann equation on the the Tamm -- Mott-Smith approximation in the theory of strong shock waves.

\subsection{General theorem about projector with preservation of dissipativity}

The answer to the main question is a combination of constructions presented by Theorems \ref{Theo:NearEq} and \ref{Theor1Dpro}. 

Let $M\subset U$ be an anzatz submanifold,  $x\in M$, $x\neq x^{\rm eq}$. Assume the {\it transversality condition}: differential of $H$ at $x$ does not annuls the tangent space $T_x(M)$. This means that $x$ is not a critical point of $H_M$. We are looking for projectors $\mathcal{P}_x:\mathbb{R}^n\to T_x(M)$ that project any dissipative vector field $Q$ ($D_x(H)(Q)\leq 0$) into a dissipative vector field on $M$ ($D_x(H_M) (\mathcal{P}_x(Q))\leq 0$). We will find the necessary conditions, prove that these conditions define unique projector, and  demonstrate that this projector transforms any dissipative vector fields on $U$ into a dissipative vector fields on $M$.

Consider the set of vectors $Q_D(x)=\{Q(x)\}$, where $Q$ is a dissipative vector field on $U$. It includes the open half-space $\{Q | D_x (H) (Q)<0 \}$ and belongs to the closed half-space $\{Q | D_x (H) (Q)\leq 0 \}$:
$$\{Q | D_x (H) (Q)<0 \} \subset  Q_D(x) \subset \{Q | D_x (H) (Q)\leq 0 \}$$
The second inclusion coincides with the dissipativity condition: the derivative of $H$ along $H$ is non-positive. For the first inclusion we have to use Proposition \ref{ProposExtend}. Use any local coordinate system in $U$ near $x$ and define a constant vector field $Q(y)=Q(x)$ in a  small vicinity of $x$ if this vicinity is sufficiently small then there $D_y (H) (Q(y))<0$. Use the extension Theorem \ref{TheorExtens} to extend this locally given dissipative vector field to the entire $U$. The first inclusion is proven.

A similar statement is valid for vector fields on $M$. Consider the set of vectors $Q_{D, M}(x)= {Q(x)}$, where $Q$ is a dissipative vector field on $M$. It includes the open half-space $\{Q | D_x (H_M) (Q)<0 \}$ and belongs to the closed half-space $\{Q | D_x (H_M) (Q)\leq 0 \}$,
$$\{Q | D_x (H_M) (Q)<0 \} \subset  Q_D(x) \subset \{Q | D_x (H_M) (Q)\leq 0 \},$$
where all derivatives are calculated on $M$.

\begin{lemma}\label{LemmaHalfSpaces}
Projector $\mathcal{P}_x$ projects the closed half-space $\{Q | D_x (H) (Q)\leq 0 \}$ in $\mathbb{R}^n$  onto the closed half-space $\{Q | D_x (H_M) (Q)\leq 0 \}$ in $T_x(M)$:
\begin{equation}\label{HalfSpaceProjection}
\mathcal{P}_x(\{Q | D_x (H) (Q)\leq 0 \}) = \{Q | D_x (H_M) (Q)\leq 0 \}
\end{equation}
\end{lemma}
\begin{proof}
According to Proposition \ref{ProposExtend}, each vector field on $M$ that satisfies the inequality  $D_x (H_M) (Q)<0$ can be extended from a vicinity of $x$ on $M$ into a dissipative vector field in entire $U$. The projection of this set of extensions is again $Q_{D, M}(x)$. Therefore projection of the half-space $\{Q | D_x (H) (Q)\leq 0 \}$ by $\mathcal{P}_x$ should coincide with the half-space $\{Q | D_x (H_M) (Q)\leq 0 \}$. Indeed, the inclusion
$$\mathcal{P}_x (\{Q | D_x (H) (Q)\leq 0 \}) \subseteq \{Q | D_x (H_M) (Q)\leq 0 \}$$ 
holds because for each dissipative vector field  $W(x)$ its projection on $T_x(M)$ should be also dissipative.
The inclusion
$$ \{Q | D_x (H_M) (Q)\}< 0 \} \subseteq  \mathcal{P}_x (\{Q | D_x (H) (Q)< 0 \})$$
holds because each vector $V$ of the open half-space $\{Q | D_x (H_M) (Q)\}$ can be extended to the dissipative vector field on $U$. The value of this field at point $x$ is  again $V\in T_x(M)$ and it projection $\mathcal{P}_x V =V$. The inclusions for the closed half-spaces follow from the proven inclusions for the open half-spaces and continuity of projectors (finite-dimensional linear operators).
\end{proof}

This lemma also gives the following equality
\begin{lemma}\label{LemmaBorderProjection}
\begin{equation}\label{BorderProjection}
\mathcal{P}_x(\{Q | D_x (H) (Q)= 0 \}) = \{V | D_x (H_M) (V)= 0 \}
\end{equation}
\end{lemma}
\begin{proof}
Indeed, if $D_x (H) (Q)= 0$ then $D_x (H) (-Q)= 0$ and projections of both vectors $Q$ and $-Q$ belong to the closed half-space $\{Q | D_x (H) (Q)\}\leq 0 \}$. Assume that $D_x (H_M)(\mathcal{P}_x (Q)) <0$. then $D_x (H_M)(\mathcal{P}_x (-Q))>0$ and $\mathcal{P}_x(-Q) \notin \{Q | D_x (H_M) (Q)\}< 0 \}$. This contradiction with (\ref{HalfSpaceProjection}) proves the inclusion.  
$$\mathcal{P}_x(\{Q | D_x (H) (Q)= 0 \}) \subseteq \{V | D_x (H_M) (V)= 0 \}.$$
Another inclusion
$$\{V | D_x (H_M) (V)= 0 \} \subseteq \mathcal{P}_x(\{Q | D_x (H) (Q)= 0 \})$$
is satisfied because $\{V | D_x (H_M) (V)= 0 \}\subset \{Q | D_x (H) (Q)= 0 \}$.  
\end{proof}

On the subspace $\ker D_x (H)=\{Q | D_x (H) (Q)= 0 \}$ the projector $\mathcal{P}_x$ acts as an orthogonal projector in the Shahshahani inner product.
\begin{lemma}\label{Restriction}
Restriction $\mathcal{P}_x$ onto $\ker D_x (H)$ is an orthogonal projector in the Shahshahani inner product $\langle \cdot | \cdot \rangle_x$:
$$\mathcal{P}^{\perp}:\ker D_x (H) \to \ker D_x (H_M)$$
\end{lemma} 
\begin{proof}
Consider restriction of $H$ onto intersection of hyperplane $x+\ker D_x (H)$ with $U$. On
$U \cap  (x+\ker D_x (H))$ point $x$ is a global minimizer of $H$. Similarly, point $x$ is a global minimizer of $H$ on $U \cap  (x+\ker D_x (H_M))$. Notice that $(U \cap  (x+\ker D_x (H_M)))\subset  (U \cap  (x+\ker D_x (H)))$. For projections of dissipative vector fields in
$U \cap  (x+\ker D_x (H))$ into dissipative vector fields in $U \cap  (x+\ker D_x (H_M))$ near point $x$ we can use Theorem \ref{Theo:NearEq}. According to this theorem, the restriction of $\mathcal{P}_x$ on $\ker D_x (H)$ is an orthogonal projector onto $D_x (H_M)$ in the Shahshahani metrics.
\end{proof}

Finally, we approach a linear algebra problem: Let $L$ be an Euclidean space with inner product $\langle \cdot | \cdot \rangle$. Assume that $W$ is a subspace in $L$, $l\neq 0$ is a linear functional in $L$, $W\not\subset \ker l$, $\mathcal{P}:L\to W$ is a projector, ${\rm im} \mathcal{P}=W$, $\mathcal{P}(\{y\in L | l(y)\leq 0\})=\{z\in W | l(z)\leq 0\}$, and the restriction of $\mathcal{P}$ on $\ker l$, $\mathcal{P}|_{\ker l}:\ker l \to (\ker l \cap W)$ is an orthogonal projector in the inner product $\langle \cdot | \cdot \rangle$. Demonstrate that these conditions define $\mathcal{P}$ unambiguously and find an explicit construction of this projector.

Let $L_0 = \ker l$ and $W_0=W \cap \ker l$. Find a unit normal vector $\nu$ to $W_0$. Space $L$ can be presented as an orthogonal sum of three subspaces:
$$L=W_0 \oplus \{y\in L_0| y\perp W_0\} \oplus \{\alpha \nu | \alpha \in R\}.$$
Subspace $W$ can be presented as an orthogonal sum of two subspaces:
$$W=W_0 \oplus \{\alpha \nu_W | \alpha \in R\},$$
where $\nu_W\in W$ is the unit normal vector to $W_0$ in $W$.
Note, that in a general situation $\nu_W \neq \nu$. For convenience, we choose such orientation of normal vectors that $l(\nu)<0$ and $l(\nu_W) <0$.

If a vector $z\in L$ is a orthogonal sum 
\begin{equation}\label{ortsum}
z= z_0\oplus z_0^{\perp}\oplus \zeta \nu  
\end{equation}
($z_0\in W_0$, $z_0^{\perp}\in \{y\in L_0| y\perp W_0\}$ and $\zeta \in R$), then its projection onto $W$ should have a form:
$$\mathcal{P}z= z_0 \oplus \beta \zeta \nu_W,$$
where $\beta$ is a real number. This coefficient is defined by the condition $\mathcal{P}(\nu_W)=\nu_W$ because projector $\mathcal{P}$ is the identity operator on its image.

Vector $\nu_W$ can be represented in $L$ as an orthogonal sum: $\nu_W= \nu_W^{\parallel}\oplus 
\nu_W^{\perp}$, where $\nu_W^{\perp}=\langle \nu_W|\nu\rangle \nu$ and  $\nu_W^{\parallel}=\nu_W-\nu_W^{\perp}$. $\nu_W^{\parallel}\in W_0$ because $\langle \nu_W^{\parallel}|\nu \rangle=0$. Therefore, $\mathcal{P}\nu_W=\beta \langle \nu_W|\nu\rangle \nu_W$.

The condition $\mathcal{P}\nu_W=\nu_W$ gives:  $\beta \langle \nu_W|\nu\rangle=1$ and
$$\beta=\frac{1}{\langle \nu_W|\nu\rangle}.$$

Finally, the projector $\mathcal{P}$ is unambiguously defined  and for the projection of the orthogonal sum (\ref{ortsum}) we obtain:
\begin{equation}\label{Projector1}
\boxed {\mathcal{P}(z_0\oplus z_0^{\perp}\oplus \zeta \nu)=z_0 \oplus \frac{\zeta \nu_W}{\langle \nu_W|\nu\rangle}} 
\end{equation}
Note that this projector preserves the value of $l(z)$. First of all, 
$$l(\nu_W)=l(\nu_W^{\perp})=\langle \nu_W|\nu\rangle l(\nu)$$
because $\nu_W^{\parallel}\in \ker l$. Note that $l(z_0\oplus z_0^{\perp}\oplus \zeta \nu)=\zeta l(\nu)$. Therefore
$$l(\mathcal{P}(z_0\oplus z_0^{\perp}\oplus \zeta \nu))=l\left(z_0 \oplus \frac{\zeta \nu_W}{\langle \nu_W|\nu\rangle}\right)=\zeta l(\nu)=l(z_0\oplus z_0^{\perp}\oplus \zeta \nu).$$

Thus, we required the preservation of the sign of $l(z)$ and obtained the preservation of its value. 
 
Lemmas \ref{LemmaHalfSpaces}, \ref{LemmaBorderProjection} and \ref{Restriction} together with formula (\ref{Projector1}) prove the following theorem. Let $x\in U$ be a non-critical point of $H$. Assume that $M\subset U$ is a smooth  immersed ansatz manifold, $x\in M$, and the transversality condition holds: $T_x(M)\not\subset \ker D_x H$ (that is, $M$ is not tangent to the level set of $H$ at point $x$). Let us use the notations $L_0=\ker D_x H$ and $W_0=T_x(M)\cap L_0$. We use the Shahshahani inner product $\langle \cdot | \cdot \rangle_x$. Find unit normal vectors $\nu$ for $L_0$ and $\nu_W \in T_x(M)$ for $W_0$. For most applications, it is convenient, to select the ``antigradient'' orientations of these normal vectors: $D_x(H)(\nu)<0$ and $D_x(H_M)(\nu_W)<0$.
Represent $L_0$ in the form of the orthogonal sum: $L_0=W_0\oplus \{y\in L_0| y\perp W_0\}$. Each vector $Q\in \mathbb{R}^n$ can be represented as an orthogonal sum: $Q=Q_0\oplus Q_0^{\perp}\oplus\zeta \nu$.
\begin{theorem}\label{TheorFinalProjector}
Projector field $\mathcal{P}_x$ that solves the problem of dimensionality reduction with preservation of dissipativity acts as follows: 
\begin{equation}\label{ProjectorFin}
\boxed{\mathcal{P}_x(Q)=Q_0\oplus\frac{\zeta \nu_W}{\langle \nu_W|\nu\rangle}} 
\end{equation}
\end{theorem} 
\begin{flushright} $\square$ \end{flushright}

The applications of Theorem \ref{TheorFinalProjector} are the same for all problems of dimension reduction: 
\begin{itemize}
\item Find Hessian of $H$ and define the Shahshahani inner product $\langle \cdot | \cdot \rangle_x$.
\item Define the unit normal  vector $\nu=\nu(x)$ to $L_0=\ker D_x (H(x))$. For each $x\in U$ this vector is the normalized antigradient $H$ in the inner product $\langle \cdot | \cdot \rangle_x$:  $\nu(x)=-\frac{{\rm grad}_x H(x)}{\|{\rm grad}_x H(x)\|}$.
\item For a given ansatz manifold $M$ and any $x\in M$ find the unit normal vector $\nu_W=\nu_W(x) \in T_x(M)$ to the subspace $W_0=T_x(M)\cap \ker D_x(H_M)$. For each $x\in M$ this vector is the normalized antigradient  $H_M$ in the restriction of the inner product $\langle \cdot | \cdot \rangle_x$ onto $T_x(M)$: $\nu_W(x)=-\frac{{\rm grad}_x H_M(x)}{\|{\rm grad}_x H_M(x)\|}$.
\item Find the orthogonal decomposition $L_0=W_0\oplus  \{y\in L_0| y\perp W_0\}$.
\item Use the projector (\ref{ProjectorFin}) for any dissipative vector field.
\end{itemize}

The only computationally expensive task in high dimensionality may be the orthogonal decomposition for a given ansatz manifold. Calculation of gradients requires several comments. Gradient of a function in a Euclidean space is the  Riesz representation of its differential and depends on the inner product. By definition, for the Shahshahani inner product $\langle \cdot | \cdot \rangle_x$, $D_x(H(x)(Q)=\langle{\rm grad}_x H(x), Q \rangle_x$. Similarly, $D_x(H_M(x)(Q)=\langle{\rm grad}_x H_M(x), Q \rangle_x$, but for ${\rm grad}_x H_M(x)$ all calculations are performed in the space $T_x(M)$ and the vector of gradient belongs to this space. In general, ${\rm grad}_x H_M(x)\neq {\rm grad}_x H(x)$. The standard coordinate representation of the gradient as the vector of partial derivatives is valid only in the orthonormal basis in the Shahshahani inner product. It may be worth mentioning that the antigradient in the Shahshahani inner product is the direction of descent in Newton's method for minimizing a function.

Indeed, let $e_x$ be a gradient $H(x)$ at point $x\in U$ with respect to the inner product
$\langle \cdot | \cdot \rangle_x$. Let us use the initial (standard) coordinates in $\mathbb{R}^n$. Recall that $\langle y | z \rangle_x=(y , {\rm Hes}(H)z)$, where ${\rm Hes}(H)$ is the Hessian matrix of second-order partial derivatives of $H$ and $(y, z )=\sum y_i z_i$ is the standard inner product in $\mathbb{R}^n$. The gradient vector $e_x$ is defined by equations:
$$\sum_{i=1}^n \frac{\partial H}{\partial x_i} Q_i=\sum_{i=1}^n Q_i \left(\sum_{j=1}^n ({\rm Hes}_x(H))_{ij} (e_x)_j\right)  \mbox{  for all  } Q.$$
Vectors $Q$ can be excluded from these equations and we obtain
$$\nabla_x H={\rm Hes}_x(H) e_x; \quad e_x=({\rm Hes}_x(H))^{-1}\nabla_x H,$$
where $\nabla_x H$ is the vector of partial derivatives of $H$ in the initial coordinate system. 
Finally, $\nu $ in the  definition of projector (\ref{ProjectorFin}) is: 
$$\nu=-\frac{e_x}{\sqrt{\langle e_x | e_x\rangle_x}}.$$ 

The similar calculations in an internal coordinate system on an ansatz manifold $M$ give  us the unit normal vector to $\ker D_x H_M(x)$,  $\nu_W\in T_x(M)$.

For convenience, in the next subsection we collect the most popular functions $H(x)$ and their Hessian matrices.

\subsection{Shahshahani inner products and gradients for several popular divergences \label{SecMorimoto}}

Many of applied Lyapunov function are particular cases of the following family:

\begin{equation}\label{Morimoto}
H(x)=\sum_i x_i^{\rm eq}f\left(\frac{x_i}{x_i^{\rm eq}}\right),
\end{equation} 
where $f$ is a convex function on the positive semi-axis.

These functions were introduced by R\'{e}nyi \cite{Renyi1961} in the same work, where he introduced the R\'{e}nyi entropy. He proved that they are Lyapunov functions for Markov kinetics. It is noteworthy that the function $H(x)$ (\ref{Morimoto}) depends only on the current value of $x$ and on the equilibrium value $x^{\rm eq}$, and does not depend on the kinetic coefficients, which are usually not as well known as the equilibrium. There exist no other Lyapunov functions for the Markov kinetics with this property (independence of kinetic constants). This characterisation of $H(x)$ was proven by P. Gorban in 2003 \cite{PGorban2003} and then independently by S.-I. Amari in 2009 \cite{Amari2009}. These functions were also rediscovered by Morimoto \cite{Morimoto1963} and studied further by Csiszar \cite{Csiszar1963}. They are known as $f$-divergences or the Csiszar-Morimoto relative entropies.

The derivatives of the $f$-divergences (\ref{Morimoto}) are
$$\frac{\partial H}{\partial x_i}=f'\left(\frac{x_i}{x_i^{\rm eq}}\right),$$
where $f'$ is the derivative of $f$.

The Hessian matrix ${\rm Hes}_x(H)$ is
\begin{equation}
\frac{\partial^2 H}{\partial x_i \partial x_j}= \delta_{ij}\frac{1}{x_i^{\rm eq}}f''\left(\frac{x_i}{x_i^{\rm eq}}\right),
\end{equation}
where $f''$ is the second derivative of $f$.
This matrix is diagonal, therefore, the corresponding Shahshahani inner product is very simple:
$$\langle z | y \rangle_x= \sum_i f''\left(\frac{x_i}{x_i^{\rm eq}}\right) \frac{z_i y_i}{x_i^{\rm eq}}.$$
The gradient of $H(x)$ with respect to this inner product is a vector $e_x$ with components
$$e_{xi}=x_i^{\rm eq} \frac{f'\left(\frac{x_i}{x_i^{\rm eq}}\right)}{f''\left(\frac{x_i}{x_i^{\rm eq}}\right)}$$

If $f(z)=z\ln z$ then the $f$-divergence (\ref{Morimoto}) is the Kullback--Leibler divergence very popular both in kinetics and machine learning:
$$H(x)=\sum_i x_i \ln \left(\frac{x_i}{x_i^{\rm eq}}\right).$$
For this function,
$$\frac{\partial H}{\partial x_i}=1+\ln \left(\frac{x_i}{x_i^{\rm eq}}\right), \quad \frac{\partial^2 H}{\partial x_i \partial x_j}= \frac{\delta_{ij}}{x_i}, \quad \langle z | y \rangle_x = \sum_{i}\frac{z_i y_i}{x_i}, \quad e_{xi}=x_i\left(1+ \ln\left(\frac{x_i}{x_i^{\rm eq}}\right)\right). $$

If $f(z)= z(\ln z-1)$ then the formulas become even simpler:
$$H(x)=\sum_i x_i \left(\ln \left(\frac{x_i}{x_i^{\rm eq}}\right)-1\right).$$
$$\frac{\partial H}{\partial x_i}=\ln \left(\frac{x_i}{x_i^{\rm eq}}\right), \quad \frac{\partial^2 H}{\partial x_i \partial x_j}= \frac{\delta_{ij}}{x_i}, \quad \langle z | y \rangle_x = \sum_{i}\frac{z_i y_i}{x_i}, \quad e_{xi}=x_i\ln\left(\frac{x_i}{x_i^{\rm eq}}\right). $$

If $f(z)=-\ln z$ then $H(x)$ is the relative Burg entropy:
$$H(x)=-\sum_i x_i^{\rm eq}\ln \left(\frac{x_i}{x_i^{\rm eq}}\right).$$
$$\frac{\partial H}{\partial x_i}=-\frac{x_i^{\rm eq}}{x_i}, \quad \frac{\partial^2 H}{\partial x_i \partial x_j}= \delta_{ij}\frac{x_i^{\rm eq}}{x_i^2}, \quad \langle z | y \rangle_x = \sum_{i}\frac{z_i y_i x_i^{\rm eq}}{x_i^2}, \quad e_{xi}=-x_i. $$

For all other examples of $f$-divergences calculations are essentially the same. Calculations  on an arbitrary subspace requires more efforts with orthogonalisation, at least, partial.

\section{Beyond manifolds: monotone trees \label{monoton}}

Approximating data sets or dynamical systems by projection onto a manifold often seems suboptimal, since a relatively high-dimensional approximation may contain too much ``air'' (empty or nearly empty space irrelevant to the applied problem), while a manifold of lower dimension and limited curvature may lose the quality of the approximation. 

The idea of data approximation by ``principal graphs and trees'' is well-elaborated and there exists an open source  software for use in many areas \cite{Albergante2020}, from single cell omics in bioinformatics to dynamic phenotyping of diseases in clinical medicine and even in astronomy for analysis of distributions of galaxies. Nevertheless, projection of dynamics onto principal graphs and trees remains an open problem \cite{GolovenBac2020, Bell2024}.

In this section, we present a solution of the problem of projection dynamics onto trees with preservation of dissipativity.  Trees in the graph-theoretical sense may be also viewed as topological spaces by representation of each edge as an arc. A tree is smoothly immersed in U if each arc is a smooth curve. In a tree, each two nodes are connected by a unique path that consists of arcs (edges). A tree is monotone with respect to function $H$ if $H$ is a strongly monotone functions along each of these paths and its derivative along edges do not vanish.

The problem of preservation of dissipativity for projection of any dissipative vector field on a monotone tree is de-facto solved in Section \ref{SSeq1D}. We should just apply Theorem \ref{Theor1Dpro} to the paths on the tree. Technically, the result is very simple. Let $T\subset U$ be a monotone tree. Find the node   $x^*\in T$ with minimal value of $H$. This node is unique because each two nodes are connected by unique path and $H$ is a strongly monotone function along these paths. Let $x$ be a point on the tree and the line $L$ be a path in $T$ between $x$ and $x^*$. Select $H$ as an internal coordinate on the $L$ ($H_L$). For every dissipative vector field $Q$ define motion on this path by the condition 
$$\frac{dH_L(x)}{dt}=D_x (H(x))(Q).$$
According to Theorem \ref{Theor1Dpro} this definition gives the unique projector of any dissipative vector fields into dissipative motion along the path $L$. Note that the path $L$ is piecewise smooth, therefore we have to apply Theorem to arcs and then glue motions in nodes.

For non-monotonic trees, the problem of flow distribution at bifurcation points arises. This problem is stated as important for applied work \cite{GolovenBac2020, Bell2024}. Here we cannot even hypothesise  a deterministic solution. Perhaps, to reduce dissipative systems to non-monotonic graphs, it is necessary to introduce stochastic models with randomized flow distribution at branching points.

\section{Conclusion and outlook \label{Concl}}

We have presented a detailed solution to the projection problem for the reduction of a dissipative system with incomplete information about vector fields. For this purpose, we have found projectors that project any dissipative vector field into a dissipative vector field on ansatz manifolds. This work is a further mathematical development of previous work on physical and chemical kinetics. The general form of projectors is universal and can be used in any field to reduce dimensionality when there is no detailed knowledge of the vector field. Further refinement of the model can be performed after reducing the dimension.  There are many possible applications, from control and nonequilibrium thermodynamics to machine learning.

The answer to the problem is presented in Theorem \ref{TheorFinalProjector}. It is simple and transparent. Calculations of the gradient of $H$ with respect to the Shahshahani inner product $\langle \cdot | \cdot \rangle_x$ in $\mathbb{R}^n$ are straightforward for many popular  Lyapunov functions that are useful both in machine learning and in non-equilibrium thermodynamics (Section \ref{SecMorimoto}). For an ansatz manifold of small dimension or codimension, calculations are also simple (see, for example, Section \ref{SSeq1D}). Nevertheless, calculations of projections for an ansatz manifold of intermediate dimension (neither small dimension, nor small codimension) require additional efforts for development of optimised algorithms.

Reduces dynamics beyond manifolds is very attractive in many applications. The algorithms for data approximation by graphs and, in particular, by trees are well developed and the open source software is  available (for its description and further references we refer to  \cite{Albergante2020}). Despite of these successes, the problem of projection of dynamics onto ansatzes that are not manifolds has been repeatedly described as important in applications and unsolved \cite{GolovenBac2020, Bell2024}. Here we have taken the first step towards solving this problem: we have introduced monotone trees and found a projection operator that projects any dissipative dynamics in space into the dissipative dynamics on the tree. The problem of projection onto non-monotonic trees and graphs requires a change in the framework of consideration: it seems necessary to consider not deterministic dynamics, but dynamics with stochastic redistribution of the phase flow at bifurcation points. Stochastic modeling is beyond the scope of this article, but is certainly the first task on the ``to do'' list.

Finally, let us repeat the main message of the paper: the projection of dissipative dynamics with preservation of dissipativity requires a special design of the projector. Surprisingly, a universal projector exists, and it is unique.

\section{Acknowledgements}

This work was supported by the Analytical Center for the Government of the Russian Federation (agreement identifier 000000D730324P540002, grant No 70-2023-001320 dated 27.12.2023).

\section*{CRediT authorship contribution statement}

{\bf Sergey V. Stasenko}: Writing – review \& editing, Investigation, Formal analysis. {\bf Alexander N. Kirdin}: Writing – review \& editing, Writing – original draft, Methodology, Investigation, Formal analysis, Conceptualization.

\section*{Declaration of competing interest}
 
 The authors declare that they have no known competing financial interests or personal relationships that could have appeared  to influence the work reported in this paper

\section*{Data availability}

No data have been used in the paper.

\end{document}